\documentclass[12pt,oneside,a4paper]{amsart}
\usepackage[left=3cm,right=3cm,top=2cm,bottom=2cm,footskip=1cm,includeheadfoot]{geometry}
\usepackage[english]{babel}
\usepackage[utf8]{inputenc}
\usepackage{fancyhdr}
\usepackage{inputenc, amsmath, amssymb, latexsym, epsfig, rotating, fancyhdr, amsthm, pifont,mathtools}
\usepackage{verbatim} 

\usepackage{latexsym, enumerate, txfonts, expdlist, hyperref, mdwlist, mathrsfs, tikz, todonotes, epsfig,  epic}
\usepackage[active]{srcltx}
\usepackage[T1]{fontenc}
\usepackage{subfig}
\newcommand{\X}{\ensuremath{X}}
\newcommand{\Sep}{\ensuremath{\operatorname{Sep}}}

\newcommand{\mc}{\mathcal}

\newcommand{\alphlist}{\begin{list}{(\alph{enumi})}{\usecounter{enumi}}}
\newcommand{\romanlist}{\begin{list}{(\roman{enumi})}{\usecounter{enumi}}}
\newcommand{\listend}{\end{list}}

\renewcommand{\:}{\colon}

\newcommand{\ssq}{\ensuremath{\subseteq}}

\newcommand{\eps}{\ensuremath{\varepsilon}}

\newcommand{\N}{\ensuremath{\mathbb{N}}} 

\newcommand{\Z}{\ensuremath{\mathbb{Z}}}

\newcommand{\C}{\ensuremath{\mathbb{C}}}

\newcommand{\diam}{\ensuremath{\mathrm{diam}}}

\newcommand{\abs}[1]{\ensuremath{\left|#1\right|}}

\theoremstyle{plain}
\newtheorem{thm}{Theorem}[section]
\newtheorem{lem}[thm]{Lemma}
\newtheorem{prop}[thm]{Proposition}
\newtheorem{cor}[thm]{Corollary}

\theoremstyle{definition}

\theoremstyle{remark}

\numberwithin{equation}{section}
\numberwithin{thm}{section}

\makeatletter
\newcommand*{\rom}[1]{\expandafter\@slowromancap\romannumeral #1@}
\makeatother
\renewcommand{\phi}{\varphi}
\renewcommand{\liminf}{\varliminf}
\renewcommand{\limsup}{\varlimsup}
\newcommand{\sub}{\vartheta}
\newcommand{\Dim}{\ensuremath{\mathrm{dim}}}
\newcommand{\htop}{\ensuremath{h_{\mathrm{top}}}}
\newcommand{\udens}{\overline{D}}

\newcommand{\dsep}{\Delta}
\renewcommand{\Sep}{\ensuremath{\mathrm{Sep}}}

\newcommand{\ac}{\ensuremath{\mathrm{{ac}}}}

\newcommand{\oac}{\ensuremath{\overline{\mathrm{ac}}}}
\newcommand{\uac}{\ensuremath{\underline{\mathrm{ac}}}}
\newcommand{\sigmaX}{\left.\sigma\right|_{X}}
\newcommand{\sigmaY}{\left.\sigma\right|_{Y}}

\newcommand{\nefw}{\ensuremath{A^+}}

\title{Constant length substitutions, iterated function systems and amorphic complexity}

\author[Gabriel Fuhrmann]{Gabriel Fuhrmann}
\author[Maik Gröger]{Maik Gröger}

\address[Gabriel Fuhrmann]{Department of Mathematics, Imperial College London, 180 Queen's Gate,
London SW7 2AZ, United Kingdom }
\email{gabriel.fuhrmann@imperial.ac.uk}

\address[Maik Gröger]{Faculty of Mathematics, University of Vienna, Oskar Morgensternplatz 1,
1090 Vienna, Austria}
\email{maik.groeger@univie.ac.at}

\thanks{This project has received funding from the European Union's Horizon 2020
research and innovation program under the Marie Sklodowska-Curie grant agreement
No 750865.
Further, MG acknowledges support by the DFG grants JA 1721/2-1 and GR 4899/1-1
and would like to thank Henna Koivusalo for very helpful discussions concerning
the theory of iterated function systems on general metric spaces.}

\begin{document}
\maketitle

\begin{abstract}
	We show how geometric methods from the general theory of fractal dimensions
	and iterated function systems can be deployed to study symbolic dynamics
	in the zero entropy regime.
	More precisely, we establish a dimensional characterization of the topological
	notion of amorphic complexity.
	For subshifts with discrete spectrum associated to constant length substitutions,
	this characterization allows us to derive bounds for the amorphic complexity
	by interpreting the subshift as the attractor of an iterated function system
	in a suitable quotient space.
	As a result, we obtain the general finiteness and positivity of amorphic
	complexity in this setting and provide a closed formula in case of a
	binary alphabet.
\end{abstract}

\section{Introduction}

The relation between the dimension theory of dynamical systems and ergodic theory
is nowadays a well-established field of research.
One of the first to formally show this relation explicitly was Billingsley
\cite{Billingsley1960}.
He proved that for an expanding circle map $f \colon x\mapsto bx\!\mod 1$ 
with an $f$-invariant measure $\mu$ the following equality holds
\[
	\Dim_H(\mu)=\frac{h_{\mu}(f)}{\log b},
\]
where $\Dim_H(\mu)$ denotes the Hausdorff dimension of the measure $\mu$ and
$h_{\mu}(f)$ the measure theoretic entropy of $f$ with respect to $\mu$.
Later, Furstenberg \cite{Furstenberg1967} established a topological analogue of
this equality, namely, $\Dim_H(K)=\htop(\left.f\right|_K)/\log b$
for any $f$-invariant set $K$ in the circle  with $\Dim_H(K)$ denoting the Hausdorff
dimension of the set $K$ and $\htop(\left.f\right|_K)$ the topological entropy of
$f$ restricted to $K$.

The relevance of these equalities originates from the fact that they relate
different concepts measuring the ``size'' of an invariant object.
Roughly speaking, entropy determines the size by measuring the amount of the supported
disorder and Hausdorff dimension specifies an actually geometric size.
In the concrete setting described above, there is an intermediary needed to bring
both notions together, represented here by the term $\log b$.
It also allows for a dynamical interpretation, namely as the Lyapunov exponent
of $f$ (see, for instance, \cite{BarreiraPesin2002} for more information regarding
this notion).
In this article we will be mainly concerned with systems in the symbolic setting
and there, no intermediary is needed as we explain in the following paragraph.

Furstenberg's result actually holds in the general context of symbolic dynamical
systems.
Suppose we are given a subshift $(X,\sigma)$ of $(A^{\Z},\sigma)$
where $A$ is a finite alphabet, $\sigma$ is the left shift and $A^{\Z}$ is equipped
with the \emph{Cantor metric} $d_\beta(x,y)=\beta^{-j}$ with $\beta>1$ and
$j=\min\{\abs{k}: x_k\neq y_k\}$ for $x=(x_k)_{k\in\Z}$
and $y=(y_k)_{k\in\Z}$ in $A^{\Z}$.
Then Furstenberg's result reads as follows
(see also \cite{Simpson2015} for a generalization to 
higher-dimensional subshifts)
\begin{align}\label{eq:rel_Hausdorffdim_entropy_symbolic_shifts}
	\Dim_H(X)=\frac{\htop(\sigmaX)}{\log\beta}.
\end{align}
Moreover, for the box dimension $\Dim_B$ of $X$ (see Section \ref{sec:Besicovitch spaces and box dimension})
it is easily seen that
\begin{align}\label{eq:rel_boxdim_entropy_symbolic_shifts}
	\Dim_B(X)=\frac{\htop(\sigmaX)}{\log\beta}.
\end{align}
Summing up, for symbolic systems, the correspondence between entropy and dimension
can be established directly.
Here, the term $\log\beta$ is just a normalizing constant, depending on the
chosen metric on $X$, and we can regard (topological) entropy and (Hausdorff and
box) dimension in the symbolic setting as one and the same quantity.

For subshifts of zero entropy, equations \eqref{eq:rel_Hausdorffdim_entropy_symbolic_shifts}
and \eqref{eq:rel_boxdim_entropy_symbolic_shifts} do not yield much insight.
However, there are many natural and interesting families of symbolic systems with zero
entropy.
For instance, Sturmian and regular Toeplitz subshifts belong to this class.
Another important subclass are substitutive subshifts including those related
to the Pisot Conjecture.
For more information concerning these systems, see for example
\cite{Fogg2002}, \cite{Kurka2003} and \cite{Queffelec2010}.
Therefore, a desirable goal for symbolic dynamics in this low complexity regime is to
find an analogously intertwined behavior of (slow entropy) concepts measuring
disorder on the one side and suitable fractal dimensions on the other side.

In this article we provide this kind of relation for the topological notion of
amorphic complexity which was recently introduced in \cite{FuhrmannGroegerJaeger2016}
to study dynamical systems with zero entropy.
More precisely, we show that amorphic complexity of symbolic systems coincides
with the box dimension of an associated subset in the so-called Besicovitch space
of $(A^{\Z},\sigma)$ endowed with a canonical metric, see
Lemma \ref{prop: euiv family metrics} and Theorem \ref{thm: amorphic complexity and box dimension}.

In fact, as we will see, this equality can be established by elementary means.
At the same time, however, it embodies a key insight which puts us in a position
where we can utilize methods from fractal geometry.
For instance, this allows us to immediately deduce that amorphic complexity of mean
equicontinuous subshifts (see Section~\ref{sec:mean equicontinuity}) is bounded
from below by the topological dimension of their maximal equicontinuous factor.
We will explain this in more detail at the end of Section \ref{sec:Besicovitch spaces and box dimension},
where we also relate this observation to substitutive subshifts associated to
substitutions of Pisot type.

The connection between amorphic complexity and box dimension in the
symbolic setting is crucial for our main results.
In particular, this relation enables us to deploy the theory of iterated
function systems on general complete metric spaces to estimate the amorphic
complexity of subshifts with discrete spectrum associated to primitive constant
length substitutions, see Section \ref{sec: ac of constant length subs}.
To state some of the results, let us briefly provide a
background on amorphic complexity and substitutive subshifts.

For now we introduce amorphic complexity for symbolic systems.
The general definition can be found in Section \ref{sec:asymp sep numbers and ac},
where we also provide a short overview of some of the basic
properties of this topological invariant.
Given a subshift $(X,\sigma)$, $x,y\in X$ and $\delta>0$, we set
\begin{align}\label{eq: def_D_deltas}
	D_{\delta}(x,y)=\limsup_{n\to \infty}\frac{\#\left\{0\leq k\leq n-1 :
		d\left(\sigma^k(x),\sigma^k(y)\right)\geq\delta \right\}}{n},
\end{align}
where $d$ can be any metric generating the product topology on $A^{\Z}$.
We define the \emph{asymptotic separation numbers} as
\begin{align*}
	\Sep(\sigmaX,\delta,\nu)=\sup\left\{\# Y : Y\subseteq X\textnormal{ with }
		D_{\delta}(x,y)\geq\nu \textnormal{ for all } x \neq y \textnormal{ in } Y\right\},
\end{align*}
with $\delta>0$ and $\nu\in(0,1]$. 
The \emph{amorphic complexity} of $(X,\sigma)$ is set to be
\begin{equation*}
	\ac\big(\sigmaX\big)=\adjustlimits\sup_{\delta>0}
		\lim_{\nu\to 0}\frac{\log\Sep(\sigmaX,\delta,\nu)}{-\log \nu},
\end{equation*}
whenever the limit in $\nu$ exists (otherwise, we define the \emph{lower}
and \emph{upper amorphic complexity} $\uac$ and $\oac$, by taking the limit inferior
and limit superior, respectively).

A \emph{substitution} over a finite alphabet A
is a map $\sub:A\to\nefw$,
where $\nefw$ is the set of all non-empty finite words with letters from $A$.
Observe that by concatenation, $\sub$ can be considered a map from
$\nefw$ to $\nefw$ and from $A^{\Z}$ to $A^{\Z}$ in the natural way.
We call a substitution $\sub$ \emph{primitive} if for some $\ell>0$ and every
$a\in A$ the word $\sub^\ell(a)$ contains each of the letters of $A$.
We say a substitution $\sub$  is of \emph{constant length} if there is $\ell\in\N$
such that $\sub(a)\in A^{\ell}$ for each $a\in A$ (and we set $\abs{\sub}=\ell$).
For every constant length substitution $\sub$ there exists
a periodic point $x_0$ of $\sub$ in $A^{\Z}$, that is, there is $p\in \N$ with 
$\sub^p(x_0)=x_0$ (see \cite{Gottschalk1963}).
Finally, if $\sub$ is primitive, we denote by $X_{\sub}$ the shift orbit
closure of some periodic point of $\vartheta$.
This notation is justified because $\X_{\vartheta}$
turns out to be independent of the particular $\sub$-periodic point.
We call $(X_{\sub},\sigma)$ a \emph{substitution} (or \emph{substitutive}) subshift.
It is well known that $X_\sub$ is minimal \cite{Gottschalk1963} and has a unique
Borel probability measure which is invariant under the action of the left shift \cite{Klein1972}.

Let us now state our main results.
We start by considering substitutions over two symbols whose associated
substitutive subshift is infinite (a subshift $(X,\sigma)$ is called \emph{finite}
if $X$ is finite; otherwise it is called \emph{infinite}).
In this case, we can establish a closed formula for the amorphic complexity of
$(X_\sub,\sigma)$.
The proof of the next statement can be found at the end of Section \ref{sec:dimensional estimates}.

\begin{thm}\label{thm: general formula ac for subs over two symbols}
	Let $\vartheta:\{0,1\}\to\{0,1\}^+$ be a primitive constant length substitution.
	Assuming that $(X_{\sub},\sigma)$ is infinite, we get
	\begin{equation*}
		\ac\big(\left.\sigma\right|_{X_{\sub}}\big)
		=\frac{\log\abs{\sub}}{\log\abs{\sub}-\log\abs{\sub}_\ast},
	\end{equation*}
	with $0<\abs{\sub}_\ast\leq \abs{\sub}$ the number of positions where $\sub(0)$
	and $\sub(1)$ differ.
\end{thm}

An interesting consequence of Theorem \ref{thm: general formula ac for subs over two symbols}
is that the amorphic complexity of infinite substitution subshifts over two symbols
with discrete spectrum is always finite and bounded from below by one.
As it turns out, this holds true over general alphabets.

\begin{thm}\label{thm: intro general lower and upper bounds for ac}
	Assume $\sub:A\to\nefw$ is a primitive substitution of constant length.
	If $(X_{\sub},\sigma)$ is infinite and has discrete spectrum, then
	\begin{equation*}
		1\leq\uac\big(\left.\sigma\right|_{X_{\sub}}\big)
		\leq\oac\big(\left.\sigma\right|_{X_{\sub}}\big)<\infty.
	\end{equation*}
\end{thm}

We want to point out that the proof of the previous theorem yields
means to compute concrete lower and upper bounds for the
amorphic complexity of infinite substitutive subshifts with discrete spectrum, see
Section \ref{sec: finiteness and positivity ac}.
Finally, using Theorem \ref{thm: intro general lower and upper bounds for ac} and further
properties of substitution subshifts together with the general theory of amorphic
complexity, we obtain the following trichotomy, see also
Section \ref{sec: finiteness and positivity ac}.

\begin{cor}
	Suppose $\sub:A\to\nefw$ is a primitive substitution of constant length.
	\begin{enumerate}[(i)]
		\item $\ac(\sigma|_{X_\sub})=0$ iff	$(X_{\sub},\sigma)$ is finite.
		\item $1\leq\uac(\sigma|_{X_\sub})\leq\oac(\sigma|_{X_\sub})<\infty$
			iff $(X_{\sub},\sigma)$ has discrete spectrum and is infinite.
		\item $\ac(\sigma|_{X_\sub})=\infty$ iff $(X_{\sub},\sigma)$ has
			partly continuous spectrum.
	\end{enumerate}
\end{cor}


\section{Basic notation and definitions}\label{sec: preliminaries}

Given a continuous self-map $f:X\to X$ on a compact metric space $(X,d)$, we call
the pair $(X,f)$ a \emph{(topological) dynamical system}.
We say $(X,f)$ is \emph{invertible} if $f$ is invertible.\footnote{For future
reference, we would like to phrase some of the preliminary results in the
non-invertible setting,  although our main results are actually concerned with
invertible systems.}
Given two dynamical systems $(X,f)$ and $(Y,g)$, we say  $(Y,g)$ is a \emph{factor}
of $(X,f)$ if there exists a continuous onto map $h:X\to Y$ such that $h\circ f=g\circ h$.
If additionally, $h$ is invertible, then $h$ is a \emph{conjugacy} and we say
$(X,f)$ and $(Y,g)$ are \emph{conjugate}.

A system $(X,f)$ is \emph{equicontinuous} if the family $(f^n)_{n\in\N}$ is
uniformly equicontinuous, that is, if for all $\nu>0$ there is $\delta>0$ such that
for all $x,y\in X$ with $d(x,y)<\delta$ we have $d(f^n(x),f^n(y))<\nu$ ($n\in\N$).
It is well known, see for instance \cite[Theorem~2.1]{Downarowicz2005}, that every
dynamical system $(X,f)$ has a unique (up to conjugacy) \emph{maximal equicontinuous factor}:
an equicontinuous factor $(Y,g)$ so that every other equicontinuous factor of
$(X,f)$ is also a factor of $(Y,g)$.

A subset $E\subseteq X$ is \emph{$f$-invariant} if it is closed and $f(E)= E$.
In this case, we call $(E,\left.f\right|_E)$ (and sometimes synonymously $E$ itself)
a \emph{subsystem} of $(X,f)$ and usually just write $(E,f)$ for notational convenience.
We say $E$ is \emph{$f$-minimal} if it is $f$-invariant and does not contain any
non-empty proper subset which is $f$-invariant.
In case that $X$ is $f$-minimal itself, we also say that $(X,f)$ is \emph{minimal}.

Given a topological dynamical system $(X,f)$, a Borel probability measure $\mu$
on $X$ is called \emph{$f$-invariant} if $\mu(f^{-1}(E))=\mu(E)$ for all Borel
measurable sets $E\ssq X$.
An invariant measure $\mu$ is called \emph{ergodic} if for all Borel measurable
$E\ssq X$ with $f^{-1}(E)=E$ we have $\mu(E)\in \{0,1\}$.
We say $(X,f)$ is \emph{uniquely ergodic} if there exists exactly one $f$-invariant
measure $\mu$.
Note that in this case, the unique invariant measure $\mu$ is ergodic.
 
We will mainly deal with bi-infinite shift spaces $\Sigma=A^\Z$, where $A$ is a
finite set also referred to as \emph{alphabet} and $\Sigma$ carries the product
topology.
We define the \emph{full shift} to be the system $(\Sigma,\sigma)$, where
$\sigma\:\Sigma\to \Sigma$ is the left shift, that is, 
$\sigma( (x_n)_{n\in\Z})=(x_{n+1})_{n\in\Z}$ for all $(x_n)_{n\in\Z}\in \Sigma$.
For $n\in \N$ elements of $A^n$ are called \emph{words of length $n$}.
We set $\nefw=\bigcup_{n\in \N}A^n$. 

Subsystems of the full shift are referred to as \emph{subshifts}.
One way to obtain subshifts is to consider \emph{orbit closures}: given $x_0\in\Sigma$,
we define its orbit closure to be
$X_{x_0}=\overline{\{\sigma^k(x_0) : k\in\Z\}}\subseteq\Sigma$.
Clearly, $X_{x_0}$ is $\sigma$-invariant.
Given a subshift $(X,\sigma)$, the collection of all words of length $n$ that
\emph{appear} in $X$ is denoted by ${\mc L}^n(X)$.
That is, $w\in {\mc L}^n(X)$ if and only if $w\in A^n$ and there is $x\in X$ such
that $w_i=x_{i}$ ($i=0,\ldots,n-1$).
Given two minimal subshifts $X\neq Y$ of $\Sigma$, it is easy to see that
there is $n\in \N$ with ${\mc L}^n(X)\cap {\mc L}^n(Y)=\emptyset$.


\section{Asymptotic separation numbers and amorphic complexity}\label{sec:asymp sep numbers and ac}

In this section, we briefly introduce amorphic complexity which is a conjugacy
invariant that is of particular relevance in the class of mean equicontinuous systems
(see Section~\ref{sec:mean equicontinuity}).
Given a dynamical system $(X,f)$, $x,y\in X$ and $\delta>0$, we set
\begin{equation*}
	\dsep(f,\delta,x,y)=\left\{k\in\N_0 : d(f^k(x),f^k(y))\geq\delta\right\}.
\end{equation*}
For $\nu\in(0,1]$ we say $x$ and $y$ are \emph{$(f,\delta,\nu)$-separated}
if $\udens(\dsep(f,\delta,x,y))\geq\nu$, where $\udens(E)$ denotes the
\emph{upper density} of a subset $E\subseteq\N_0$ defined as
\begin{equation*}
	\udens(E)=\varlimsup\limits_{n\to\infty}\frac{\#(E\cap[0,n-1])}{n}.
\end{equation*}
A subset $S\subseteq X$ is said to be \emph{$(f,\delta,\nu)$-separated} if all pairs of
distinct points $x,y\in S$ are $(f,\delta,\nu)$-separated.  
The \emph{(asymptotic) separation numbers} of $(X,f)$, denoted by $\Sep(f,\delta,\nu)$
for $\delta>0$ and $\nu\in(0,1]$, are defined as the largest cardinality of
an $(f,\delta,\nu)$-separated set contained in $X$.
If $\Sep(f,\delta,\nu)$ is finite for all $\delta>0,\nu\in(0,1]$, we say $(X,f)$ has
{\em finite separation numbers}, otherwise we say it has {\em infinite separation numbers}.

As the next theorem suggests, finite separation numbers correspond to systems
of low dynamical complexity, that is, in particular to systems of zero entropy
and those without non-trivial weakly mixing measures (see, e.g., \cite{Walters1982}
for definitions of these concepts).

\begin{thm}[\cite{FuhrmannGroegerJaeger2016}]
	If a dynamical system $(X,f)$ has positive topological entropy or is weakly
	mixing with respect to some invariant probability measure $\mu$ with non-trivial
	support, then $(X,f)$ has infinite separation numbers.
\end{thm}

For systems with finite separation numbers, we may obtain further quantitative
information by studying the scaling behavior of the separation numbers as the
separation frequency $\nu$ goes to zero.
We define the \emph{lower} and \emph{upper amorphic complexity} of $(X,f)$ as
\begin{equation*}
	\uac(f)=\adjustlimits\sup_{\delta>0}\varliminf_{\nu\to 0}
		\frac{\log \Sep(f,\delta,\nu)}{-\log \nu}
	\quad\textnormal{and}\quad
	\oac(f)=\adjustlimits\sup_{\delta>0}\varlimsup_{\nu\to 0}
		\frac{\log\Sep(f,\delta,\nu)}{-\log \nu}. 
\end{equation*}
If both values coincide, we call $\ac(f)=\uac(f)=\oac(f)$ the {\em amorphic
complexity} of $(X,f)$. 
Observe that by allowing the above quantities to assume values
in $[0,\infty]$, they are actually well defined for any map $f:X\to X$.
In particular, systems with infinite separation numbers have infinite amorphic
complexity.

We refer the reader to \cite{FuhrmannGroegerJaeger2016} for an in-depth discussion
of amorphic complexity and several classes of examples.
However, the following statement is worth recalling.
\begin{thm}[\cite{FuhrmannGroegerJaeger2016}]\label{thm: properties ac}
	Let $(X,f)$ and $(Y,g)$ be dynamical systems.
	\begin{enumerate}[(i)]
			\item Suppose $(Y,g)$ is a factor of $(X,f)$.
					Then $\oac(f)\geq\oac(g)$ and $\uac(f)\geq \uac(g)$.
					In particular, amorphic complexity is an invariant of topological
					conjugacy.
			\item Let $m\in\N$. 
		Then $\oac(f^m)=\oac(f)$ and
					$\uac(f^m)=\uac(f)$.
	\end{enumerate}
\end{thm}

Since our main results involve invertible dynamical systems, let us close this
section with a comment on the following issue.
If $f$ is invertible, it may seem natural to define
asymptotic separation numbers and amorphic complexity in a slightly different way:
for $\delta>0$, $\nu\in(0,1]$ and $x,y\in X$  we could consider
\begin{equation*}
	\dsep'(f,\delta,x,y)=\left\{k\in\Z : d(f^k(x),f^k(y))\geq\delta\right\}
\end{equation*}
instead of $\dsep(f,\delta,x,y)$.
Thus, we may consider a modified upper density for
subsets $E\in\Z$,
\begin{equation*}
	\udens'(E)=\varlimsup\limits_{n\to\infty}\frac{\#(E\cap[-n+1,n-1])}{2n-1},
\end{equation*}
to define the $(f,\delta,\nu)$-separation of two points $x$ and $y$ and hence
the asymptotic separation numbers  as well as the amorphic complexity (where all
subsequent  definitions from this section would carry over in the obvious way).
However, in \cite{FuhrmannGroegerJaegerKwietniak2019} it is shown that for a
natural class of invertible systems (which includes mean equicontinuous systems)
this neither affects whether $(X,f)$ has infinite separation numbers nor the value
of the (lower and upper) amorphic complexity.
For that reason, we stick with the above definitions.


\section{Mean equicontinuity and finite separation numbers}\label{sec:mean equicontinuity}

In this section, we discuss a canonical class of systems with finite separation
numbers which in particular comprises all substitutive subshifts associated to
primitive substitutions with discrete spectrum (see Section~\ref{sec: discrete spectrum}).
We say a dynamical system $(X,f)$ is \emph{(Besicovitch-) mean equicontinuous}
if for every $\nu>0$ there is $\delta>0$ such that for all $x,y\in X$ with
$d(x,y)<\delta$ we have
\[
	D_B(x,y)=
		\varlimsup\limits_{n\to\infty}\frac{1}{n}\sum\limits_{k=0}^{n-1}d(f^k(x),f^k(y))<\nu.
\]
This notion was introduced by Li, Tu and Ye in \cite{LiTuYe2015}.
It is immediately seen to be equivalent to the concept of mean Lyapunov-stability
which was introduced in 1951 by
Fomin \cite{Fomin1951} in the context of systems with discrete spectrum.
A first systematic treatment is due to Auslander \cite{Auslander1959}.
For recent activity related to these notions,
see for example \cite{DownarowiczGlasner2016,Garcia-Ramos2017,FanJiang2018,
QiuZhao2018,FuhrmannGroegerLenz2018,HuangWangZhang}.

It is worth noting that $D_B$ is a pseudometric.
With this in mind, it is not hard to see that $(X,f)$ is mean equicontinuous if
and only if  $D_B\: X\times X \to [\vphantom{]}0,\infty)$ is continuous.
Further, it is immediate that
\begin{align}\label{eq: besicovitch metric and density of separation}
	\delta\cdot\udens(\dsep(f,\delta,x,y))\leq D_B(x,y)\qquad (x,y\in X),
\end{align}
where $\delta>0$.
Since $X$ is compact, this yields the next statement.
\begin{prop}\label{prop: mean_equicont_finite_sep_num}
	If $(X,f)$ is mean equicontinuous, then it has finite separation numbers.
\end{prop}

In the minimal case we can say even more.
To that end, let us introduce the following notion: a system $(X,f)$ is called
\emph{mean sensitive} if there is $\nu>0$ such that for every $x\in X$ and
$\delta>0$ there exists $y\in X$ with $d(x,y)<\delta$ and $D_B(x,y)>\nu$.

\begin{prop}[{\cite[Corollary~5.5 \& Proposition~5.1(4)]{LiTuYe2015}}]\label{prop: mean_sens_inf_sep_num}
	A minimal dynamical system $(X,f)$ is either mean equicontinuous or mean sensitive.
	Moreover, if $(X,f)$ is mean sensitive, then $(X,f)$ has infinite separation numbers.
\end{prop}

As an immediate corollary of Proposition~\ref{prop: mean_equicont_finite_sep_num}
and Proposition~\ref{prop: mean_sens_inf_sep_num}, we obtain the following assertion.

\begin{thm}\label{thm:equivalence_mean_equi_fin_sep_num}
	Suppose $(X,f)$ is minimal.
	Then $(X,f)$ is mean equicontinuous if and only if $(X,f)$ has finite separation numbers.
\end{thm}

Let us briefly come back to the discussion at the end of the previous section.
If $f$ is invertible, we may consider the pseudometric
\[
	D_B'(x,y)=
		\varlimsup\limits_{n\to\infty}\frac{1}{2n-1}\sum\limits_{k=n-1}^{n-1}d(f^k(x),f^k(y))
		\qquad (x,y\in X),
\]
instead of $D_B$.
However, the results of \cite{DownarowiczGlasner2016} and \cite{QiuZhao2018,FuhrmannGroegerLenz2018}
(the former treats the minimal case and the latter treat the general case)
yield that $D_B$ is continuous if an only if $D_B'$ is continuous. 
We may hence stick with $D_B$ in the following.


\section{Discrete spectrum and mean equicontinuity}\label{sec: discrete spectrum}

Suppose $(X,f)$ is a dynamical system and let $\mu$ be an $f$-invariant measure.
We call $\lambda\in \C$ an \emph{eigenvalue} and $\varphi \in L_2(X,\mu)$
a corresponding \emph{eigenfunction} of $(X,f)$ if $\varphi\circ f=\lambda\cdot\varphi$.
Further, $(X,f)$ is said to have \emph{(purely) discrete spectrum} with respect to $\mu$ if
there is an orthonormal basis for $L_2(X,\mu)$ consisting of eigenfunctions.

\begin{thm}[{\cite[Corollary~6.3]{FuhrmannGroegerLenz2018}}]
	Assume $(X,f)$ is an invertible minimal dynamical system.
	Then $(X,f)$ is mean equicontinuous if and only if it is uniquely ergodic,
	has discrete spectrum and each eigenvalue possesses a continuous eigenfunction.
\end{thm}

For subshifts associated to primitive substitutions, a classical result by Host
\cite{Host1986} (see \cite[Theorem 6.3 \& Remark 6.3(1)]{Queffelec2010}, too)
states that all eigenvalues possess a continuous eigenfunction.
Hence, the previous theorem immediately yields the following corollary, where
we make use of Theorem \ref{thm:equivalence_mean_equi_fin_sep_num}
and the fact that the subshift $(X_\sub,\sigma)$ associated to a primitive substitution
$\sub:A\to\nefw$ is always minimal and uniquely ergodic.

\begin{cor}\label{cor:discrete spectrum iff finite sep numbers}
	Let $\sub$ be a primitive substitution and $(X_{\sub},\sigma)$ the associated
	subshift.
	Then the following are equivalent.
	\begin{enumerate}[(i)]
		\item $(X_{\sub},\sigma)$ is mean equicontinuous.
		\item $(X_{\sub},\sigma)$ has discrete spectrum.
		\item $(X_{\sub},\sigma)$ has finite separation numbers.
	\end{enumerate}
\end{cor}


\section{Box dimension and amorphic complexity}\label{sec:Besicovitch spaces and box dimension}

In this section, we establish a connection between the amorphic complexity of a
subshift and the box dimension of an associated subset in the so-called Besicovitch space.

We start by briefly introducing the box dimension of a totally bounded subset $E$
of a general metric space $(M,\rho)$.
We call a subset $S$ of $M$ \emph{$\eps$-separated}
if for all $s\neq s'\in S$ we have $\rho(s,s')\geq\eps$ and denote
by $M_\eps(E)$ the maximal cardinality of an $\eps$-separated subset of $E$.
Then the \emph{lower} and \emph{upper box dimension} of $E$ are defined as
\begin{align}\label{eq: definition box dimension}
	\underline\Dim_B(E)=\varliminf\limits_{\eps\to 0}
		\frac{\log M_\eps(E)}{-\log\eps}
	\qquad\textnormal{ and }\qquad
	\overline\Dim_B(E)=\varlimsup\limits_{\eps\to 0}
		\frac{\log M_\eps(E)}{-\log\eps}.
\end{align}
In case that $\underline\Dim_B(E)$ and $\overline\Dim_B(E)$ coincide, their common
value, denoted by $\Dim_B(E)$, is called the \emph{box dimension}\footnote{
Let us remark that the standard definition of the box dimension involves the smallest
number of sets with diameter strictly smaller than $\eps$ needed to cover $E$. 
It is well known that the above definitions do not change if in
\eqref{eq: definition box dimension}, $M_\eps(E)$ is replaced by this number
(see also, for example, Proposition~1.4.6 in \cite{Edgar1998}).} of $E$.
We will make use of the following basic facts.

\begin{thm}\label{thm: properties box dimension}\mbox{}
	\begin{itemize}
		\item[(i)] The lower and upper box dimension are invariant with respect to
			Lip\-schitz-con\-tinuous homeomorphisms with a Lip\-schitz-continuous inverse.
		\item[(ii)] Given a metric space $(M,\rho)$, a totally bounded subset
			$E\ssq M$ and a Lipschitz-continuous map  $h:E\to M$.
			Then we have
			\[
				\underline\Dim_B(E)=\underline\Dim_B\left(\bigcup_{i=0}^{n-1} h^i(E)\right)
				\leq\overline\Dim_B\left(\bigcup_{i=0}^{n-1}h^i(E)\right)
				=\overline\Dim_B(E).
			\]
	\end{itemize}
\end{thm}
\begin{proof}
	We only show the first equality in part (ii).
	The second equality in (ii) as well as part (i) can be shown analogously, see
	also \cite[Theorem 6.3 \& Appendix I]{Pesin1997}.

	It is straightforward to see that
	\begin{align*}
		\max_{0\leq i<n}\underline\Dim_B(E_i)\leq\underline\Dim_B\left(\bigcup_{i=0}^{n-1} E_i\right)
	\end{align*}
	for totally bounded $E_0,E_1,\ldots, E_{n-1} \ssq M$ (see also
	\cite[Theorem 6.2  \& Appendix I]{Pesin1997}).
	It hence suffices to show 
	$\underline\Dim_B\left(\bigcup_{i=0}^{n-1} h^i(E)\right)\leq\underline\Dim_B(E).$
	To that end, let $L>1$ be a mutual Lipschitz constant of the maps
	$h,\ldots, h^{n-1}$.
	We clearly have $M_{L\eps}(h^i(E))\leq M_{\eps}(E)$ for all $\eps>0$ and each
	$i\in\{0,\ldots,n-1\}$.
	Hence,
	\[
		\underline\Dim_B\left(\bigcup_{i=0}^{n-1} h^i(E)\right)\leq 
		\varliminf\limits_{\eps\to 0} \frac{\log \sum_{i=0}^{n-1} M_{L\eps}(h^i(E))}{-\log L\eps}
		\leq\varliminf\limits_{\eps\to 0} \frac{\log \left(n\cdot M_{\eps}(E)\right) }{-\log L\eps}=
		\underline\Dim_B(E). \qedhere
	\]
\end{proof}

The first part of the above statement has an immediate but important corollary:
if $\rho$ and $\rho'$ are \emph{Lipschitz-equivalent} metrics on $M$, that is, if there
are $c,C>0$ such that
\[
	c\cdot\rho(x,y)\leq \rho'(x,y)\leq C\cdot\rho(x,y) \qquad (x,y\in M),
\]
then the box dimension of any totally bounded subset $E$ of $M$ is independent of whether we 
compute $M_\eps(E)$ with respect to $\rho$ or $\rho'$.

Next, we aim at introducing the Besicovitch space associated to the full shift $(\Sigma,\sigma)$.
Recall that the definition of the quantity $D_\delta$ in \eqref{eq: def_D_deltas}
implicitly depends on the particular metric $d$ we put on $\Sigma$.
For a plethora of metrics, it turns out that $D_\delta$ is in fact a
pseudometric and moreover, Lipschitz-equivalent to any other such pseudometric
obtained from \eqref{eq: def_D_deltas}.
For the next statement and its proof, we write $D^d_\delta$ to stress the
dependence on $d$ explicitly.
Given a metric $d$ on $\Sigma$, let $\delta_0^d>0$ be such that $d(x,y)\geq \delta_0^d$
whenever $x_0\neq y_0$ where $x=(x_k)_{k\in\Z}$ and $y=(y_k)_{k\in\Z}$ are from $\Sigma$.

\begin{lem}\label{prop: euiv family metrics}
	Suppose $d$ is an ultrametric which induces the product topology on $\Sigma$
	and let $\delta>0$.
	Then $D_\delta^d$ is a pseudometric on $\Sigma$.
	
	Further, assume $d$ and $d'$ are two metrics which induce the product topology
	on $\Sigma$.
	Let $\delta\in(0,\delta_0^d]$ and $\delta'\in(0,\delta_0^{d'}]$.
	Then there are $c,C>0$ such that
	\begin{align}\label{eq: defn lipschitz equivalent pseudometrics}
		c\cdot D^d_{\delta}(x,y)\leq D^{d'}_{\delta'}(x,y)\leq C\cdot D^d_{\delta}(x,y)
		\qquad (x,y\in \Sigma).
	\end{align}
\end{lem}
\begin{proof}
	For the first part, we restrict to proving the triangle inequality since the other
	properties are trivial.
	Observe that for arbitrary $x,y,z\in\Sigma$ we have
	\begin{align*}
		&\#\left\{0\leq k\leq n-1 \colon d\left(\sigma^k(x),\sigma^k(y)\right)\geq\delta \right\}\\
		&\leq 
		\#\left\{0\leq k\leq n-1 \colon d\left(\sigma^k(x),\sigma^k(z)\right)\geq\delta
			\textnormal{ or } d\left(\sigma^k(z),\sigma^k(y)\right)\geq\delta \right\}\\
		&\leq
		\#\left\{0\leq k\leq n-1 \colon d\left(\sigma^k(x),\sigma^k(z)\right)\geq\delta \right\}
			+\#\left\{0\leq k\leq n-1 \colon d\left(\sigma^k(z),\sigma^k(y)\right)\geq\delta \right\},
	\end{align*}
	using the ultrametric inequality in the first step.
	This yields that $D_\delta$ is a pseudometric.
	
	For the second part,  we first prove that for all $\eta,\eta' \in (0,\delta_0^d]$
	there are $c,C>0$ such that
	\begin{align}\label{eq: lipschitz equivalent pseudometrics}
		c\cdot D^d_{\eta'}(x,y)\leq D^d_{\eta}(x,y)\leq C\cdot D^d_{\eta'}(x,y)
		\qquad (x,y\in \Sigma)
	\end{align}
	where we consider $D^d_{\eta'}$ and $D^d_{\eta}$ as defined in \eqref{eq: def_D_deltas}
	regardless of whether this yields pseudometrics on $\Sigma$ or not.
	Without loss of generality, we may assume $\eta'=\delta_0^d$.
	
	Note that since $d$ induces the product topology on $\Sigma$, there exists
	$m_\eta\in \N$ such that $d(x,y)\geq\eta$ yields the existence of
	$\ell\in\{-m_\eta,\ldots,m_\eta\}$ with $x_\ell\neq y_\ell$.
	Hence,
	\begin{align*}
		&\#\left\{0\leq k\leq n-1 : d\left(\sigma^k(x),\sigma^k(y)\right)
			\geq\eta \right\}
		\leq\#\left\{0\leq k\leq n-1 : x_{[k-m_\eta,k+m_\eta]}
			\neq y_{[k-m_\eta,k+m_\eta]} \right\}\\
		&\leq(2m_\eta+1)\cdot \#\left\{-m_\eta \leq k\leq n-1+m_\eta :
			x_{k}\neq y_{k} \right\}\\
		&\leq (2m_\eta+1)\cdot \#\left\{-m_\eta \leq k\leq n-1+m_\eta :
			d\left(\sigma^k(x),\sigma^k(y)\right)\geq\delta_0^d \right\}.
	\end{align*}
	Therefore, for all $x,y\in \Sigma$ we have the right-hand side of
	\eqref{eq: lipschitz equivalent pseudometrics} with $C=2m_\eta+1$.
	Further, we trivially have $D_{\delta_0^d}^d(x,y)\leq D_{\eta}^d(x,y)$ and hence 
	the left-hand side of \eqref{eq: lipschitz equivalent pseudometrics} with $c=1$.
	
	Finally, let $d$, $d'$, $\delta$ and $\delta'$ be as in the statement.
	Since $d$ and $d'$ induce the same topology, there is $\eta\in(0,\delta_0^{d}]$
	such that $d'(x,y)\geq \delta'$ implies $d(x,y)\geq\eta$.
	With such $\eta$, we have for all $x,y\in\Sigma$
	\[
		D_{\delta'}^{d'}(x,y)\leq D_{\eta}^{d}(x,y)\leq C\cdot D_{\delta}^{d}(x,y),
	\]
	for some $C>0$ which exists due to \eqref{eq: lipschitz equivalent pseudometrics}.
	By a similar argument, we obtain $D_{\delta}^{d}(x,y)\leq 1/c\cdot D_{\delta'}^{d'}(x,y)$
	for some $c>0$.
	This finishes the proof.
\end{proof}

Following a standard procedure, we introduce an equivalence relation on $\Sigma$
by identifying $x,y\in\Sigma$ whenever $D_{\delta}(x,y)=0$ and denote by
$[\cdot]$ the corresponding quotient mapping.
Note that Lemma~\ref{prop: euiv family metrics} ensures that this relation is
actually an equivalence relation and that it is well defined, i.e., independent
of the particular metric $d$ on $\Sigma$ and independent of $\delta\in(0,\delta_0^d]$.
Sometimes, we may consider subshifts with different alphabets.
In those cases, we may use the notation $[\cdot]_{A}$ to
unambiguously specify the corresponding quotient mapping (recall that $A$ denotes
the alphabet of $\Sigma$, see Section \ref{sec: preliminaries}).

For the rest of this work, we may assume without loss of generality that $d$ is
a Cantor metric (so that $\delta_0^d=1$).
Since $d$ is hence an ultrametric, $D_\delta$ is a pseudometric on $\Sigma$ for each
$\delta\in (0,1]$, which we may further consider a metric on $[\Sigma]$ in the canonical way.
We call $([\Sigma], D_\delta)$ the \emph{Besicovitch space}.

Since most of the analysis carried out in this work takes place in the space
$([\Sigma], D_\delta)$,  we provide the next statement to familiarize the reader
with its topological and metric properties.
Its proof is a straightforward application of the results in \cite{BlanchardFormentiKurka1997}.

\begin{thm}
	The space $([\Sigma], D_{\delta})$ is complete and pathwise connected.
	Moreover, it is topologically infinite dimensional and neither locally compact
	nor separable.
\end{thm}
\begin{proof}
	First, recall that it suffices to show the statement for $([\Sigma], D_1)$,
	see Lemma \ref{prop: euiv family metrics}.
	Analogously to the definition of $D_1$, we may consider
	\[
		D_{1}'(x,y)=\limsup_{n\to \infty}\frac{\#\left\{n-1\leq k\leq n-1 :
			x_k\neq y_k \right\}}{2n-1} \qquad (x,y\in\Sigma).
	\]
	Similarly as in Lemma~\ref{prop: euiv family metrics}, we see that $D_1'$
	defines a pseudometric (see also \cite{BlanchardFormentiKurka1997}) and hence
	an equivalence relation by identifying points $x,y\in\Sigma$ whenever $D_1'(x,y)=0$.
	Let us denote the corresponding quotient mapping by $[\cdot]'$ and, as before,
	consider $D'_1$ a metric on $[\Sigma]'$ in the canonical way.
	
	Now, \cite[Proposition~1--3]{BlanchardFormentiKurka1997} state that
	$([\Sigma]', D'_{1})$ is complete, pathwise connected, topologically infinite
	dimensional and neither locally compact nor separable.
	We simply carry over these properties to $([\Sigma], D_{1})$
	by means of the map
	\begin{align*}
		\phi : [\Sigma]'\to [\Sigma] &: [(\ldots,x_{-2},x_{-1},x_0,x_1,x_2,\ldots)]'\mapsto
			[(\ldots,x_{-2},x_{-1},x_0,x_{-1},x_1,x_{-2},x_2,\ldots)].
	\end{align*}
	Note that $\phi$ is isometric and surjective.
	The statement follows.
\end{proof}

Now, observe that, given a subshift $(X,\sigma)$, we have
$D_\delta(x,y)=\udens(\dsep(\sigmaX,\delta,x,y))$ for all $\delta\in(0,1]$ and $x,y\in X$.
In other words,
\begin{align*}
	\Sep(\sigmaX,\delta,\nu)=\sup\left\{\# Y : Y\subseteq X\text{ with }
		D_{\delta}(x,y)\geq\nu \text{ for all } x \neq y \text{ in } Y\right\}
\end{align*}
as stated in the introduction.
The proof of the next statement is hence straightforward.

\begin{prop}\label{prop: finite sep numbers iff projection totally bounded}
	Given a subshift $(X,\sigma)$ and $\delta\in (0,1]$.
	Then $\Sep(\sigmaX,\delta,\nu)$ is finite
	for all $\nu\in(0,1]$ if and only if $[X]$ is totally bounded in $([\Sigma],D_\delta)$.
\end{prop}

The following theorem shows that studying the amorphic complexity of a subshift
$\left(X,\sigma\right)$ amounts to studying the box dimension of $[X]$.

\begin{thm}\label{thm: amorphic complexity and box dimension}
	Suppose $(X,\sigma)$ is a subshift with finite separation numbers.
	Then the (lower and upper) box dimension of the subset $[X]$ of $[\Sigma]$
	(equipped with $D_\delta$) is independent of $\delta\in(0,1]$.
	Further,
	\begin{align*}
		\uac\big(\sigmaX\big)=\underline\dim_B\left(\left[X\right]\right)
		\quad\textnormal{and}\quad
		\oac\big(\sigmaX\big)=\overline\dim_B\left(\left[X\right]\right).
	\end{align*}
\end{thm}
\begin{proof}
	Fix $\delta\in(0,1]$.
	Since $(X,\sigma)$ has finite separation numbers, the lower and upper box
	dimension of $[X]$ in $([\Sigma],D_\delta)$ are well defined, according to
	Proposition \ref{prop: finite sep numbers iff projection totally bounded}.
	A priori, their values may still depend on $\delta$.
	Yet, using the Lipschitz-equivalence of the $D_\delta$'s (see
	Lemma \ref{prop: euiv family metrics}) and the Lipschitz-invariance of the
	box dimension (see Theorem \ref{thm: properties box dimension}(i)), we have
	that the lower and upper box dimension of $[X]$ in $([\Sigma],D_\delta)$ are
	independent of $\delta$.
	This yields the first part of the statement.
	
	For the second part, observe that for each $\nu\in(0,1]$ we have
	$\Sep(\sigmaX,\delta,\nu)=M^\delta_\nu([X])$, where $M^\delta_\nu([X])$ denotes
	the maximal cardinality of a $\nu$-separated subset of $[X]$ equipped with $D_\delta$.
	Then, 
	\begin{align*}
		\uac\big(\sigmaX\big)
		&=\sup_{\delta>0}\liminf_{\nu \to 0}\frac{\log \Sep(\sigmaX,\delta,\nu)}{-\log \nu}
		=			\sup_{\delta>0}\liminf_{\nu \to 0}\frac{\log M^\delta_\nu([X])}{-\log \nu}
		=\sup_{\delta>0}\underline\dim_B([X])\\
		&=\underline\dim_B([X]).
	\end{align*}
	Clearly, a similar statement holds for the upper amorphic complexity.
\end{proof}

Due to the previous statement, we are now in a position to study the
amorphic complexity from a geometric perspective.
For instance, we can understand the topological invariance of amorphic complexity
(of symbolic systems) also from the perspective of the following lemma.

\begin{lem}\label{lem: conjugacy between power shifts is bi-Lipschitz}
	Assume that $X\ssq A^\Z$ and $Y\ssq B^\Z$ are $\sigma^n$- and $\sigma^m$-invariant,
	respectively, with $n,m\in \N$.
	If $h$ is a factor map from $(X,\sigma^n)$ to $(Y,\sigma^m)$, then
	\begin{align}\label{eq: factor map Besicovitch space}
		[x]_{A}\mapsto[h(x)]_{B}
	\end{align}
	is a Lipschitz-continuous map from $[X]_{A}$ to $[Y]_{B}$
	equipped with the metric $D_1$.
	In particular, if $h$ is a conjugacy, then \eqref{eq: factor map Besicovitch space}
	is a Lipschitz-continuous homeomorphism with a Lipschitz-continuous inverse.
\end{lem}
\begin{proof}
	Choose $\delta>0$ small enough such that for all $y,y'\in Y$ we have that if
	there is $s\in\{-m/2,\ldots,m/2\}$ with $d\left(\sigma^{s}(y),\sigma^{s}(y')\right)\geq 1$,
	then  $d\left(y,y'\right)\geq\delta$.
	Observe that
	\[
		\begin{split}
			&\udens\left(\dsep(\sigma^m,\delta,y,y')\right)
			=\varlimsup\limits_{\ell\to\infty}1/\ell\cdot \#\left\{0\leq k\leq \ell-1 :
				d(\sigma^{km}(y),\sigma^{km}(y'))\geq\delta\right\}\\
			&\geq \varlimsup\limits_{\ell\to\infty}1/\ell\cdot \#\left\{0\leq k\leq \ell-1 :
				\exists s \in \{-m/2,\ldots,m/2\} \textnormal{ s.t. }
				d(\sigma^{km+s}(y),\sigma^{km+s}(y'))\geq1\right\}
		\end{split}
	\]
	which yields
	\begin{align}\label{eq: 1}
		\udens\left(\dsep(\sigma^m,\delta,y,y')\right)\geq\udens(\dsep(\sigma,1,y,y'))
		=D_1(y,y').
	\end{align}
	Next, because $h$ is uniformly continuous, there is $\eta\in(0,1]$ such that for all $x,x'\in X$ we have
	that $d(h(x),h(x'))\geq\delta$ implies $d(x,x')\geq\eta$.
	Therefore,
	\begin{align}\label{eq: 2}
		\begin{split}
			&\udens(\dsep(\sigma^m,\delta,h(x),h(x')))\\
				&=\varlimsup\limits_{\ell\to\infty}1/\ell\cdot \#\left\{0\leq k\leq \ell-1 : 
				d(\sigma^{km}(h(x)),\sigma^{km}(h(x')))\geq\delta\right\}\\
				&=\varlimsup\limits_{\ell\to\infty}1/\ell\cdot \#\left\{0\leq k\leq \ell-1 : 
				d(h(\sigma^{kn}(x)),h(\sigma^{kn}(x')))\geq\delta\right\}\\
				&\leq 
				\varlimsup\limits_{\ell\to\infty}1/\ell\cdot \#\left\{0\leq k\leq \ell-1 : 
				d(\sigma^{kn}(x),\sigma^{kn}(x'))\geq\eta\right\}=
			\udens(\dsep(\sigma^n,\eta,x,x')).
		\end{split}
	\end{align}
	Finally, choose $\eta'$ small enough such that for all $x,x'\in X$ we have
	that $d(x,x')\geq\eta$ implies $d(\sigma^s(x),\sigma^s(x'))\geq \eta'$ for
	all $s\in\{-n/2,\ldots,n/2\}$.
	Observe that $D_{\eta'}(x,x')=\udens(\dsep(\sigma,\eta',x,x'))
		\geq\udens(\dsep(\sigma^n,\eta,x,x'))$.
	Together with \eqref{eq: 1} and \eqref{eq: 2}, this yields
	\[
		D_{\eta'}(x,x')\geq \udens(\dsep(\sigma^n,\eta,x,x'))\geq 
		\udens(\dsep(\sigma^m,\delta,h(x),h(x')))\geq D_1(h(x),h(x')).
	\]	
	This proves that the map $[x]_{A}\mapsto[h(x)]_{B}$ is well defined and
	Lipschitz-continuous as a map from $([X]_{A},D_{\eta'})$ to $([Y]_{B},D_1)$.
	Using Lemma \ref{prop: euiv family metrics}, the statement follows.
\end{proof}

By means of Theorem \ref{thm: properties box dimension}(i) and
Theorem \ref{thm: amorphic complexity and box dimension} the previous statement
immediately yields the following fact.

\begin{cor}\label{cor: invariance amorphic complexity}
	If $(X,\sigma)$ and $(Y,\sigma)$ are conjugate, then
	\[
		\uac(\sigmaX)=\uac(\sigmaY)
		\quad\textnormal{and}\quad
		\oac(\sigmaX)=\oac(\sigmaY).
	\]
\end{cor}

Let us point out that, given a constant length substitution $\sub:A\to\nefw$,
we may also consider it a map on $[\Sigma]$ by putting $\sub([x])=[\sub (x)]$.
Note that for all $[x],[y]\in[\Sigma]$
\[
	D_1(\sub([x]),\sub([y]))\leq D_1([x],[y]).
\]
In a similar way, we can consider the left shift a mapping on the
Besicovitch space by setting $\sigma ([x])=[\sigma(x)]$.
Observe that the shift becomes an isometry on $[\Sigma]$.

\begin{prop}\label{prop: equivalence D_delta and D_B and MEF}
	A subshift $(X,\sigma)$ is mean equicontinuous if and only if the map
	$\left.[\cdot]\right|_X$ is continuous.
	In this case, $([X],\sigma)$ is the maximal equicontinuous factor of $(X,\sigma)$.
\end{prop}
\begin{proof}
	For the first part, recall that $(X,\sigma)$ is mean equicontinuous if and only if
	$D_B$ is continuous.
	It hence suffices to show that $D_B$ and $D_1$ induce the same topology, that is,
	for all $(x_n)_{n\in\N}\in X^\N$ and $x\in X$ we have
	\begin{align}\label{eq: equivalence D1 and DB}
	D_1(x_n,x) \stackrel{n\to\infty}{\longrightarrow}0 \qquad \text{if and only if} \qquad 
	D_B(x_n,x) \stackrel{n\to\infty}{\longrightarrow}0.
	\end{align}
	To that end, recall that
	\[
		D_1(x_n,x)=\udens(\dsep(\sigmaX,1,x_n,x))\leq D_B(x_n,x),
	\]
	see \eqref{eq: besicovitch metric and density of separation}.
	This shows the ``if''-part of \eqref{eq: equivalence D1 and DB}.
	To see the ``only if'', suppose	$D_1(x_n,x)\stackrel{n\to\infty}{\longrightarrow} 0$.
	Let $\eps>0$ and choose $n_\eps\in\N$ such that $D_\eps(x_n,x)<\eps$ for $n>n_\eps$.
	Note that this is possible since $D_\eps$ induces the same topology as $D_1$.
	Then we obtain for all $n>n_\eps$ that
        \[
		D_B(x_n,x)\leq D_\eps(x_n,x)+\eps\cdot (1-D_\eps(x_n,x))\leq 2\eps.
        \]
	Since $\eps>0$ was arbitrary, this shows $D_B(x_n,x)\stackrel{n\to\infty}{\longrightarrow} 0$
	and hence \eqref{eq: equivalence D1 and DB}.

	Concerning the second part, observe that the first part and the discussion before the statement show that 
	$[\cdot]$ is actually a factor map from $(X,\sigma)$ onto $([X],\sigma)$. 
	Furthermore, we have that $[x]=[y]$ if and only if $D_B(x,y)=0$.
	Now, if $q$ denotes the factor map of the mean equicontinuous system $(X,\sigma)$
	onto its maximal equicontinuous factor, it is well known that $q(x)=q(y)$ if
	and only if $D_B(x,y)=0$ (see for example \cite[Theorem 6]{Auslander1959} or
	\cite[Proposition 49]{Garcia-Ramos2017}).
	Hence, $[\cdot]$ and $q$ identify the same points.
	This proves the statement.
\end{proof}

Now, let us recall the fact that the box dimension is always bounded from below by the
topological dimension (see for instance \cite[Section~3.1]{Edgar1998}).
With this in mind, the previous proposition has an immediate consequence: it implies
that the amorphic complexity of a mean equicontinuous subshift is bounded from
below by the topological dimension of its maximal equicontinuous factor.

This observation may be of particular interest for substitutive subshifts
corresponding to  Pisot type substitutions over two
symbols (for a definition of substitutions of Pisot type, see for example Section
1.2.5 in \cite{Fogg2002}).
According to \cite{CanteriniSiegel2001}, all substitutions of Pisot type are
primitive.
Moreover, all subshifts over two symbols associated to these substitutions have
discrete spectrum, see \cite{BargeDiamond2002} and \cite{HollanderSolomyak2003}.
Hence, we can use Corollary \ref{cor:discrete spectrum iff finite sep numbers}
to deduce that the corresponding subshifts are mean equicontinuous so that we can apply
the above observation.

Two well-known examples belonging to the class of binary substitutions of Pisot
type are the Fibonacci and Tribonacci substitution.
For these two substitutions we obtain the concrete lower bound for the amorphic
complexity of one and two, respectively.
This holds true because the maximal equicontinuous factor of the former is the
circle and that of the latter is the two-torus (see, again, \cite{Fogg2002} for
more information).


\section{Amorphic complexity of constant length substitution subshifts}
\label{sec: ac of constant length subs}

From now on, $\sub:A\to\nefw$ will always denote a primitive constant length
substitution.
Further, we set $D=D_1$.
Recall that to each $\sub$ we can associate a minimal and uniquely ergodic
subshift, denoted by $(X_{\sub},\sigma)$.
The goal of this section is to examine the amorphic complexity of $(X_{\sub},\sigma)$
by studying its projection in the Besicovitch space $([\Sigma],D)$.


\subsection{Iterated function systems and cyclic partitions}\label{sec:IFS and cyclic partitions}

A finite family of contracting maps $(\varphi_i)_{i\in\{0,\ldots,\ell-1\}}$ on a
complete metric space $(M,\rho)$ is referred to as an \emph{iterated function system (IFS)}.
Given such an IFS, it is well known that there exits a unique compact subset
$K\ssq M$, called the \emph{attractor} of the IFS, fulfilling
\begin{align}\label{eq:IFS condition attractor}
	K=\bigcup\limits_{i=0}^{\ell-1} \varphi_i(K).
\end{align}
If, additionally, this union is disjoint, we say the IFS satisfies the
\emph{strong separation condition (SSC)}.
For more information on IFS's, see for instance \cite{Hutchinson1981,RajalaVilppolainen2011}.

In the following, we construct suitable iterated function systems
that we will use to estimate the amorphic complexity of infinite subshifts with
discrete spectrum associated to constant length substitutions.
A key ingredient for doing this will be the notion of cyclic $\sigma^n$-minimal
partitions.

Let $(X,f)$ be a minimal dynamical system.
We say a partition of $X$ is \emph{cyclic} if it partitions $X$ into disjoint
subsets $X_0,\ldots,X_{m-1}\subseteq X$, $m\in\N$ with $f(X_{i-1})=X_{i}$ for
$1\leq i<m$ and $f(X_{m-1})=X_0$.
A cyclic partition is called \emph{$f^n$-minimal} if each of its elements is $f^n$-minimal.
Note that such a partition is unique (up to cyclic permutation of its members).
Given an $f^n$-minimal set $X_0\ssq X$, it is easy to see that there is $1\leq m\leq n$
such that $X_0,f(X_0),\ldots,f^{m-1}(X_{0})$ is a (possibly trivial) cyclic $f^n$-minimal
partition of $X$, see \cite[Lemma 15]{Kamae1972}.
For each $n\in\N$ we denote by $\gamma(n)$ the number of elements of a cyclic
$f^n$-minimal partition of $X$.

Now, coming back to substitutive subshifts, observe that
\[
	\sub\circ\sigma=\sigma^{\abs{\sub}}\circ\sub.
\]
In other words, $\sub$ is a factor map from $(X_\sub,\sigma)$ to $(\sub(X_\sub),\sigma^{|\sub|})$.
Hence, the minimality of $(X_\sub,\sigma)$ implies that $\sub(X_{\sub})$ is
$\sigma^{\abs{\sub}}$-minimal.
Since each word that appears in $\sub(X_\sub)$ also appears in $X_\sub$,
we clearly have $\sub(X_\sub)\ssq X_\sub$, see also \cite[Lemma 3]{Kamae1972}.
Further, if $(X_{\sub},\sigma)$ is infinite, it turns out that
$\gamma(\abs{\sub})=\abs{\sub}$, see \cite[Lemma 7]{Dekking1978}.
Hence, in this case, we obtain the cyclic $\sigma^{\abs{\sub}}$-minimal partition
\begin{align}\label{eq:cyclic shift-minimal partition}
	X_{\sub}=\bigsqcup\limits_{i=0}^{\abs{\sub}-1}(\sigma^i\circ\sub)(X_{\sub}).
\end{align}

Our goal is to carry over \eqref{eq:cyclic shift-minimal partition} to the
Besicovitch space and to provide criteria which ensure that the maps
$\{\sigma^i\circ\sub\}_{i\in\{0,\ldots,\abs{\sub}-1\}}$ are contractions on $([\Sigma],D)$.
To that end, we need the following auxiliary statement.

\begin{prop}\label{prop: 0 besicovitch distance implies equality of subshifts}
	Let $(X,\sigma)$ and $(Y,\sigma)$ be two minimal subshifts of $(\Sigma,\sigma)$.
	If $[X]\cap [Y]\neq \emptyset$, then we have that $X=Y$.
\end{prop}
\begin{proof}
	Assume for a contradiction that $X\neq Y$.
	Then there exists $n\in\N$ such that $\mc L^n(X)\cap \mc L^n(Y)=\emptyset$
	(see Section~\ref{sec: preliminaries}).
	Hence, for all $x\in X$ and $y\in Y$ and all $\ell\in\N$ we obtain
	\[
		1/(\ell n) \cdot \#\{0\leq k \leq \ell n -1 : x_k\neq y_k\}\geq 1/n,
	\]
	so that $D(x,y)\geq 1/n$ which contradicts the assumptions.
\end{proof}

\begin{lem}\label{lem: shift-minimal partition}
	Suppose $(X,\sigma)$ is a subshift and $X_0,\ldots,X_{m-1}\subseteq X$ is a
	cyclic $\sigma^n$-minimal partition of $X$ for some $m,n\in\N$.
	Then
	\begin{align}\label{eq: shift-minimal partition}
		[X]=\bigsqcup\limits_{i=0}^{m-1}[X_i]=\bigsqcup\limits_{i=0}^{m-1}\sigma^{i}\left([X_0]\right).
	\end{align}
\end{lem}
\begin{proof}
	Let $q\:\Sigma\to\Sigma'=(A^n)^\Z$ denote the canonical higher power code of
	$\Sigma$, see for example Section 1.4 in \cite{LindMarcus1995}.
	Observe that $q$ is a conjugacy between $(\Sigma,\sigma^n)$ and $(\Sigma',\sigma)$.
	Hence, $(q(X),\sigma)$ is a subshift of $(\Sigma',\sigma)$ which is conjugate
	to $(X,\sigma^n)$.

	Moreover, we have that $(q(X_i),\sigma)$ ($i=0,\ldots,m-1$) is minimal, so that
	Proposition~\ref{prop: 0 besicovitch distance implies equality of subshifts}
	implies $[q(X_i)]_{A^n}\cap [q(X_j)]_{A^n}=\emptyset$
	if $i\neq j$.

	Finally, according to Lemma~\ref{lem: conjugacy between power shifts is bi-Lipschitz},
	$q$ can be considered a homeomorphism from
	$[X]_{A}$ to $[q(X)]_{A^n}$ with $q([x]_{A})=[q(x)]_{A^n}$ for $x\in X$.
	Therefore,
	\[
		q([X_i]_{A}\cap [X_j]_{A})
		=q([X_i]_{A})\cap q([X_j]_{A})
		=[q(X_i)]_{A^n}\cap [q(X_j)]_{A^n}=\emptyset,
	\]
	whenever $i\neq j$.
	Hence, $[X_i]_{A}\cap [X_j]_{A}=\emptyset$ which
	shows the first equality in \eqref{eq: shift-minimal partition}.

	For the second equality, note that $X_i=\sigma^i(X_0)$, since $X_0,\ldots,X_{m-1}$ is
	cyclic.
	The statement follows, since $[\sigma^i(X_0)]=\sigma^i([X_0])$.
\end{proof}

Next, we provide a criterion to ensure that the maps
$\{\sigma^i\circ\sub\}_{i\in\{0,\ldots,\abs{\sub}-1\}}$ are indeed contractions.
To that end, set
\[
	c_{a,b}(\sub)=\#\{0\leq j\leq\abs{\sub}-1 : \sub(a)_j=\sub(b)_j\},
\]
where $a\neq b\in A$ and define
\[	
	c(\sub)=\min_{a\neq b} c_{a,b}(\sub)\quad\quad\textnormal{and}\quad\quad
	C(\sub)=\max_{a\neq b} c_{a,b}(\sub).
\]
Note that $0\leq c(\sub)\leq C(\sub)\leq\abs{\sub}$.
If there is no risk of ambiguity, we may omit the argument $\sub$ in the following.

\begin{prop}\label{prop: bounds on contraction rate}
	Let $[x],[y]\in[\Sigma]$.
	Then
	\[
		\frac{\abs{\sub}-C}{\abs{\sub}}\cdot D([x],[y])
		\leq D(\sub([x]),\sub([y]))
		\leq\frac{\abs{\sub}-c}{\abs{\sub}}\cdot D([x],[y]).
	\]
\end{prop}
\begin{proof}
	Given $x,y\in \Sigma$, note that
	\begin{align*}
		&D(\sub(x),\sub(y))=\limsup_{n\to \infty}\frac1{|\sub|\cdot n}
			\cdot\#\{0\leq k\leq |\sub|\cdot n-1 :\sub(x)_k\neq \sub(y)_k\}\\
		&=\limsup_{n\to\infty}\frac1{|\sub|\cdot n}\sum_{a\neq b \in A}
			\#\left\{0\leq k \leq n-1 : (x_k,y_k)=(a,b)\right \} \cdot (|\sub|-c_{a,b})\\
		&\leq \frac{|\sub|-\min_{a\neq b} c_{a,b}}{|\sub|} \cdot \limsup_{n\to\infty}\frac1n
			\sum_{a\neq b \in A} \#\left\{0\leq k \leq n-1\: 
			(x_k,y_k)=(a,b)\right \}\\
		&= \frac{|\sub|-c}{|\sub|}\cdot D(x,y).
	\end{align*}
	Hence, we obtain the second inequality.
	The other inequality follows analogously.
\end{proof}

Note that the estimates from the previous proposition actually hold for all
$\sigma^i\circ\sub$, since the left shift $\sigma$ acts as an isometry on $([\Sigma],D)$.
In particular, each of these maps is a contraction on $([\Sigma],D)$ if and only if
$\sub$ is a contraction on $([\Sigma],D)$.

\begin{prop}\label{prop: contraction iff c>0}
	We have that $\sub:[\Sigma]\to[\Sigma]$ is a contraction if and only if $c>0$.
\end{prop}
\begin{proof}
	The ``if''-part immediately follows from Proposition~\ref{prop: bounds on contraction rate}.
	For the other direction, suppose there are $a\neq b\in A$ with $c_{a,b}=0$.
	Then
	\[
		D((\ldots,a,a,\ldots),(\ldots,b,b,\ldots))
		=D(\sub((\ldots,a,a,\ldots)),\sub((\ldots,b,b,\ldots)))
	\]
	and the statement follows.
\end{proof}


\subsection{Dimensional estimates and discrete spectrum}\label{sec:dimensional estimates}

In this section, we obtain lower and upper bounds for the amorphic complexity of
$(X_{\sub},\sigma)$ provided that the collection
$\{\sigma^i\circ\sub\}_{i\in\{0,\ldots,\abs{\sub}-1\}}$ (or some closely related
collection of mappings, see Theorem \ref{thm: general estimates ac for one-to-one subs with coincidence})
is an IFS and that $(X_{\sub},\sigma)$ has discrete spectrum.

Our strategy is as follows: first, we show that the attractor of the above IFS is
$[X_\sub]$ and that the IFS fulfills the strong separation condition. 
Then, we utilize general estimates for the box dimension of an IFS
attractor (and apply Theorem~\ref{thm: amorphic complexity and box dimension}).
We would like to recall that the assumption of discrete spectrum is natural:
if the spectrum of $(X_{\sub},\sigma)$ is not purely discrete, then its amorphic
complexity is infinite (see Corollary \ref{cor:discrete spectrum iff finite sep numbers}).

\begin{thm}\label{thm:discrete spectrum properties IFS}
	Suppose that $(X_{\sub},\sigma)$ is infinite, has discrete spectrum and that $c>0$.
	Then $\{\sigma^i\circ\sub\}_{i\in\{0,\ldots,\abs{\sub}-1\}}$ is an IFS which
	satisfies the SSC and whose attractor is $[X_{\sub}]$.
\end{thm}
\begin{proof}
	First, Proposition~\ref{prop: contraction iff c>0} yields that 
	$\{\sigma^i\circ\sub\}_{i\in\{0,\ldots,\abs{\sub}-1\}}$ is an IFS since $c>0$.
	Second, due to Corollary \ref{cor:discrete spectrum iff finite sep numbers},
	$(X_{\sub},\sigma)$ is mean equicontinuous.
	Hence, Proposition \ref{prop: equivalence D_delta and D_B and MEF} gives that
	the map $X_\sub\ni x\mapsto [x]$  is continuous so that its image $[X_\sub]$
	is compact.
	
	Third, consider the cyclic $\sigma^{\abs{\sub}}$-minimal partition from
	\eqref{eq:cyclic shift-minimal partition}.
	By Lemma \ref{lem: shift-minimal partition},
	we have
	\begin{align}\label{eq: [X] is attractor of IFS}
		\big[X_{\sub}\big]
		=\bigsqcup\limits_{i=0}^{\abs{\sub}-1}\sigma^i\left(\big[\sub(X_{\sub})\big]\right)
		=\bigsqcup\limits_{i=0}^{\abs{\sub}-1}(\sigma^i\circ\sub)\left(\big[X_{\sub}\big]\right),
	\end{align}
	where we used that $\sub$ commutes with $[\cdot]$ 
	(see the discussion before Proposition~\ref{prop: equivalence D_delta and D_B and MEF}).
	
	Altogether, \eqref{eq: [X] is attractor of IFS} shows that the attractor of
	$\{\sigma^i\circ\sub\}_{i\in\{0,\ldots,\abs{\sub}-1\}}$ is $[X_{\sub}]$
	and that the iterated function system fulfills the SSC.
\end{proof}

For the next theorem recall that $0\leq c\leq C\leq\abs{\sub}$ and note that if
$c=\abs{\sub}$, then the substitution subshift $(X_{\sub},\sigma)$ is finite.
This is true because in this case the substitution $\sub$ has a unique fixed
point given by $(\ldots,\sub(a),\sub(a),\ldots)$ with $a\in A$ arbitrary.

\begin{thm}\label{thm:estimates ac for good discrete constant length subs}
	Suppose $(X_{\sub},\sigma)$ is infinite, has discrete spectrum and $c>0$.
	We have
	\[
		\oac\big(\left.\sigma\right|_{X_{\vartheta}}\big)
		\leq\frac{\log\abs{\sub}}{\log\abs{\sub}-\log(\abs{\sub}-c)}.
	\]
	Furthermore, if $C\neq\abs{\sub}$, then
	\[
		\uac\big(\left.\sigma\right|_{X_{\vartheta}}\big)
		\geq\frac{\log\abs{\sub}}{\log\abs{\sub}-\log(\abs{\sub}-C)}.
	\]
\end{thm}

In the proof of Theorem \ref{thm:estimates ac for good discrete constant length subs},
we make use  of the following statement.\footnote{Let us remark that, under the
assumptions of Proposition \ref{prop: IFS lower bound Hausdorff dim}
the statement in \cite{RajalaVilppolainen2011} actually yields an estimate
$\Dim_H(K)\geq\log\ell/(-\log s)$, where $\Dim_H(K)$ denotes the Hausdorff
dimension of the attractor $K$.
However, the Hausdorff dimension is always a lower bound for the box dimension.
For the definition of $\Dim_H$ and its relation to $\Dim_B$, see for instance
\cite{Pesin1997}.}

\begin{prop}[{\cite[Proposition 4.10]{RajalaVilppolainen2011}}]\label{prop: IFS lower bound Hausdorff dim}
	Let $(\varphi_i)_{i\in\{0,\ldots,\ell-1\}}$ be an IFS on a complete
	metric space $(M,\rho)$ which fulfills the SSC and
	\[
		\rho(\varphi_i(x),\varphi_i(y))\geq s\cdot\rho(x,y)
		\qquad (x,y\in M,i\in\{0,\ldots,\ell-1\}),
	\]
	for some $s\in(0,1)$.
	Suppose $K\subseteq M$ is the attractor of the IFS.
	Then
	\begin{align}\label{eq: box dimension instead of Hausdorff}
		\underline\Dim_B(K)\geq\log\ell/(-\log s).
	\end{align}
\end{prop}

\begin{proof}[Proof of Theorem \ref{thm:estimates ac for good discrete constant length subs}]
	By means of Theorem~\ref{thm: amorphic complexity and box dimension}, we
	obtain bounds on the upper and lower amorphic complexity of $(X_\sub,\sigma)$,
	by providing bounds on the upper and lower box dimension of $[X_\sub]$.
	Due to Theorem~\ref{thm:discrete spectrum properties IFS}, $[X_\sub]$ is the
	attractor of the IFS $\{\sigma^i\circ\sub\}_{i\in\{0,\ldots,\abs{\sub}-1\}}$.
	For that reason, we will obtain the above bounds as a result of the application
	of general estimates on the box dimension of attractors of IFS's.
	
	To that end, let us discuss a simple (and standard) technique to obtain
	an upper bound for the box dimension of an attractor $K$ of
	an IFS $(\varphi_i)_{i\in\{0,\ldots,\ell-1\}}$ on a general complete metric space
	$(M,\rho)$.
	Let $r\in(0,1)$ denote a common contraction rate for $\varphi_0,\ldots,\varphi_{\ell-1}$.
	Observe that a successive application of \eqref{eq:IFS condition attractor}
	yields that for all $n\in\N$ we have
	\begin{align}\label{eq: attractor ifs successive application of IFS}
		K=\bigcup\limits_{(i_0,\ldots,i_{n-1})\in\{0,\ldots,\ell-1\}^n}
		\varphi_{i_0}\circ\ldots\circ\varphi_{i_{n-1}}(K).
	\end{align}
	Clearly, $\diam(\varphi_{i_0}\circ\ldots\circ\varphi_{i_{n-1}}(K))\leq r^n\cdot\diam(K)$,
	where $\diam(E)$ denotes the diameter of a set $E\subseteq M$.
	Hence, \eqref{eq: attractor ifs successive application of IFS} implies that
	for each $n\in \N$ the attractor $K$ can be covered by $\ell^n$ sets of diameter
	less or equal $r^n\cdot\diam(K)$.
	This yields $\overline\Dim_B(K)\leq\log\ell/(-\log r)$.

	Under the assumptions (and using the same notation) of Proposition~\ref{prop: IFS lower bound Hausdorff dim}, 
	we altogether have the following general bounds
	\begin{align}\label{eq: IFS dimensions}
		\overline\Dim_B(K)\leq\log\ell/(-\log r)
		\qquad\textnormal{and}\qquad
		\underline\Dim_B(K)\geq\log\ell/(-\log s).
	\end{align}
	
	Now, for our iterated function system $\{\sigma^i\circ\sub\}_{i\in\{0,\ldots,\abs{\sub}-1\}}$
	we have $\ell=\abs{\sub}$.
	Further, due to Proposition~\ref{prop: bounds on contraction rate}, we get
	$r=(\abs{\sub}-c)/\abs{\sub}$ and $s=(\abs{\sub}-C)/\abs{\sub}$.
	The statement follows from \eqref{eq: IFS dimensions}.
\end{proof}

We want to close this section with an application of Theorem \ref{thm:estimates ac for good discrete constant length subs}
in the case of a binary alphabet.
Beforehand, we will relate its assumptions to well-established properties of
(constant length) substitutions and restate it accordingly, see
Theorem \ref{thm: general estimates ac for one-to-one subs with coincidence}.

Let us first define the \emph{height} of $\sub$ (see also \cite[Lemma 10]{Dekking1978}) as
\[
	h(\sub)=\max\{n\in\N : n\textnormal{ and }\abs{\sub}
		\textnormal{ are coprime and }\gamma(n)=n\}.
\]
We call $\sub$ \emph{pure} if $h(\sub)=1$.
Note that if the subshifts associated to two primitive constant length substitutions
of the same length are conjugate, then their heights agree.
We say the substitution $\sub$ admits a \emph{coincidence (of order $k$)} if there
exist $k\in\N$ and $0\leq j<\abs{\sub}^k$ such that $\#\{\sub^k(a)_j : a\in A\}=1$.

\begin{lem}\label{prop:contraction iff coincidence}
	The map $\sub^\ell:[\Sigma]\to[\Sigma]$ is a contraction for some $\ell\in\N$
	(i.e., $c(\sub^\ell)>0$) if and only if the substitution $\sub$ has a coincidence.
\end{lem}
\begin{proof}
	Both directions are direct consequences of Proposition~\ref{prop: contraction iff c>0}.
	We just show the ``only if''-part.
	To that end, assume for a contradiction that
	$\sub^\ell$ is not a contraction for any $\ell\in \N$.
	Due to Proposition \ref{prop: contraction iff c>0}, this implies
	\[
		\forall\ell\in \N \, \exists a\neq b\: c_{a,b}(\sub^\ell)=0
	\]
	so that $\sub$ does not have a coincidence.
\end{proof}

In view of the previous statement, the next assertion emphasizes that the IFS approach
is tailor-made for constant length substitution subshifts with discrete spectrum.

\begin{thm}[{\cite[Theorems 4 and 7]{Dekking1978}}]\label{thm: pure sub discrete iff coincidence}
	Suppose $\sub$ is pure.
	Then $\sub$ has a coincidence if and only if $(X_{\sub},\sigma)$ has discrete
	spectrum.
\end{thm}

Now, a substitution $\sub$ is said to be \emph{one-to-one} if it is one-to-one as a
mapping from $A$ to $\nefw$, that is, $\sub(a)\neq\sub(b)$ whenever $a\neq b\in A$.
Note that if $\sub$ is one-to-one, then $C\neq\abs{\sub}$ and $\sub^k$ is also
one-to-one for each $k\in\N$.

\begin{thm}\label{thm: general estimates ac for one-to-one subs with coincidence}
	Assume that $\sub$ is one-to-one and has a coincidence of order $k$ as well
	as that $(X_{\sub},\sigma)$ is infinite.
	Setting $\Theta=\sub^k$, we have
	\[
		1\leq\frac{\log\abs{\Theta}}{\log\abs{\Theta}-\log(\abs{\Theta}-C(\Theta))}
		\leq\uac\big(\left.\sigma\right|_{X_{\sub}}\big)
		\leq\oac\big(\left.\sigma\right|_{X_{\sub}}\big)
		\leq\frac{\log\abs{\Theta}}{\log\abs{\Theta}-\log(\abs{\Theta}-c(\Theta))}<\infty.
	\]
\end{thm}
\begin{proof}
	First, observe that $\Theta:[\Sigma]\to[\Sigma]$ is a contraction, since it
	has a coincide, that is, $0<c(\Theta)$.
	Further, $C(\Theta)<\abs{\Theta}$ because $\Theta$ inherits the
	injectivity of $\sub$.
	Moreover, the substitution $\Theta$ is clearly primitive.
	By means of Theorem \ref{thm:estimates ac for good discrete constant length subs},
	we obtain the above bounds for the lower and upper amorphic complexity of
	$\left.\sigma\right|_{X_\Theta}$.
	Now, since periodic points of $\sub$ are also periodic points of $\Theta$,
	we have $X_\sub=X_{\Theta}$.
	The statement follows.
\end{proof}

If $\sub$ is not one-to-one, we can make use of the following statement by
Blanchard, Durand and Maass.

\begin{lem}[{\cite[Section 2.2]{BlanchardDurandMaass2004}}]\label{lem: one-to-one sub}
	Suppose $\sub$ is not one-to-one and $(X_{\sub},\sigma)$ is infinite.
	Then there exists a primitive constant length substitution $\sub':A'\to (A')^+$
	such that $\sub'$ is one-to-one, $\abs{\sub'}=\abs{\sub}$
	and $(X_{\sub'},\sigma)$ is conjugate to $(X_{\sub},\sigma)$.
	In particular, in case that $\sub$ is pure, then $\sub'$ is pure as well.
\end{lem}

In the situation of the previous statement, Corollary \ref{cor: invariance amorphic complexity}
yields that
\begin{align}\label{eq: equiality of ac for one-to-one sub}
	\uac\big(\left.\sigma\right|_{X_{\sub}}\big)=\uac\big(\left.\sigma\right|_{X_{\sub'}}\big)
	\quad\textnormal{and}\quad
	\oac\big(\left.\sigma\right|_{X_{\sub}}\big)=\oac\big(\left.\sigma\right|_{X_{\sub'}}\big).
\end{align}
In principle, this allows us to compute the amorphic complexity by the above methods
even if $\sub$ is not one-to-one. We make use of this fact in the next section.

Finally, we take a look at substitutions over two symbols and establish a
closed formula for the amorphic complexity of the associated subshifts.

\begin{cor}\label{cor: general formula ac for subs over two symbols}
	Let $\vartheta:\{0,1\}\to\{0,1\}^+$.
	Suppose $(X_{\sub},\sigma)$ is infinite and has discrete spectrum.
	Then
	\begin{equation*}
		\ac\big(\left.\sigma\right|_{X_{\sub}}\big)
		=\frac{\log\abs{\sub}}{\log\abs{\sub}-\log\abs{\sub}_\ast},
	\end{equation*}
	where $\abs{\sub}_\ast\in (0,\abs{\sub})$ denotes the number of positions where $\sub(0)$
	and $\sub(1)$ differ.
\end{cor}
\begin{proof}
	First, observe that $c=C=\abs{\sub}-\abs{\sub}_\ast$.
	Since $(X_{\sub},\sigma)$ is infinite, $\sub$ is one-to-one and
	$\abs{\sub}_\ast>0$ (see the short discussion before
	Theorem \ref{thm:estimates ac for good discrete constant length subs}).
	Furthermore, we claim that $\sub$ is pure.
	This holds because $1\leq h(\sub)\leq\#A=2$ and since $h(\sub)=\#A$ implies
	that $(X_{\sub},\sigma)$ is finite, see \cite[Remark 9(i) \& (iii)]{Dekking1978}.
	Therefore, $\sub$ has a coincidence, according to Theorem \ref{thm: pure sub discrete iff coincidence}.
	Note that this coincidence must be of order $1$ and hence, $\abs{\sub}_\ast<\abs{\sub}$.
	Finally, we obtain the desired formula for the amorphic complexity from the
	general estimates provided in 
	Theorem \ref{thm: general estimates ac for one-to-one subs with coincidence}.
\end{proof}


\subsection{Finiteness and positivity of amorphic complexity}\label{sec: finiteness and positivity ac}

We finally turn to general constant length substitutions whose associated subshifts
have discrete spectrum.
It is an interesting observation that Corollary \ref{cor: general formula ac for subs over two symbols}
implies that the amorphic complexity of infinite subshifts with discrete spectrum
associated to binary substitutions is always finite and bounded from below by one.
It turns out that this fact holds true in general.

\begin{thm}\label{thm: general lower and upper bounds for ac}
	Suppose $(X_{\sub},\sigma)$ is infinite and has discrete spectrum.
	Then
	\begin{equation*}
		1\leq\uac\big(\left.\sigma\right|_{X_{\sub}}\big)
		\leq\oac\big(\left.\sigma\right|_{X_{\sub}}\big)<\infty.
	\end{equation*}
\end{thm}

We will need the following lemma to prove Theorem \ref{thm: general lower and upper bounds for ac}.

\begin{lem}[{\cite[Lemma~17 \& 19]{Dekking1978}}]\label{lem: pure base}
	Assume that $\sub:A\to\nefw$ is not pure.
	Then there exists a $\sigma^{h(\sub)}$-minimal subset $X_0\subset X_{\sub}$
	and a pure primitive constant length
	substitution $\eta:B\to B^+$ such that
	$(X_{\eta},\sigma)$ is conjugate to $(X_{0},\sigma^{h(\sub)})$.
\end{lem}

The substitution $\eta$ from Lemma \ref{lem: pure base} is also called the
\emph{pure base} of $\sub$.

\begin{proof}[Proof of Theorem \ref{thm: general lower and upper bounds for ac}]
	First, suppose $\sub$ is pure.
	W.l.o.g.\ we may assume that $\sub$ is one-to-one because
	of Lemma \ref{lem: one-to-one sub} and \eqref{eq: equiality of ac for one-to-one sub}.
	Further, due to Theorem \ref{thm: pure sub discrete iff coincidence}, $\sub$
	has a coincidence of order $k\in\N$.
	By Theorem \ref{thm: general estimates ac for one-to-one subs with coincidence},
	we get $\uac(\sigma|_{X_\sub})\geq1$ and $\oac(\sigma|_{X_\sub})<\infty$.

	Now, assume that $\sub$ is not pure.
	Consider the $\sigma^{h(\sub)}$-minimal subset $X_0\subset X_{\sub}$ provided
	by Lemma \ref{lem: pure base}.
	By definition of $h(\sub)$, we get that $X_0,\sigma(X_0),\ldots,\sigma^{h(\sub)-1}(X_0)$
	is a cyclic $\sigma^{h(\sub)}$-minimal partition of $X_\sub$.
	Therefore, Lemma \ref{lem: shift-minimal partition} yields 
	\[
		\big[X_{\sub}\big]=\bigsqcup\limits_{i=0}^{h(\sub)-1}\sigma^i\big(\big[X_0\big]\big).
	\]
	Further, because $\sigma$ is an isometry on $([\Sigma],D)$, we can use
	Theorem~{\ref{thm: properties box dimension}(ii)} to obtain
	\begin{align*}
		\underline\Dim_B\big(\big[X_0\big]\big)
		=\underline\Dim_B\big(\big[X_\sub\big]\big)
		\leq\overline\Dim_B\big(\big[X_\sub\big]\big)
		=\overline\Dim_B\big(\big[X_0\big]\big).
	\end{align*}

	Finally, since $(X_{\eta},\sigma)$ is conjugate to $(X_{0},\sigma^{h(\sub)})$,
	Lemma \ref{lem: conjugacy between power shifts is bi-Lipschitz} yields
	a Lip\-schitz-con\-tinuous homeomorphism from $([X_\eta]_{B},D)$
	to $([X_0]_{A},D)$ with a Lip\-schitz-continuous inverse.
	By Theorem {\ref{thm: properties box dimension}(i)} and
	Theorem \ref{thm: amorphic complexity and box dimension}, we hence have
	\begin{equation*}
		\uac\big(\left.\sigma\right|_{X_{\eta}}\big)
		=\uac\big(\left.\sigma\right|_{X_{\sub}}\big)
		\leq\oac\big(\left.\sigma\right|_{X_{\sub}}\big)
		=\oac\big(\left.\sigma\right|_{X_{\eta}}\big).
	\end{equation*}
	Since $\eta$ is pure, the statement follows from the first part of the proof.
\end{proof}

Let us point out that the proof of Theorem \ref{thm: general lower and upper bounds for ac}
and its ingredients
in principle yield an algorithm to compute concrete bounds for the amorphic
complexity of $(X_{\sub},\sigma)$:
\begin{enumerate}
	\item Follow the proof of \cite[Theorem~14]{Dekking1978}, to determine the
		pure base $\eta:B\to B^+$ of $\sub$ (which coincides with $\sub$ if $\sub$
		is pure).
	\item  If $\eta$ is not one-to-one, follow the method described in 
		\cite[Section 2.2]{BlanchardDurandMaass2004} to obtain the one-to-one
		substitution $\eta':B'\to (B')^+$ (see Lemma~\ref{lem: one-to-one sub}).
		Otherwise, set $\eta'=\eta$.
		Recall that in any case $\eta'$ is pure.
	\item Apply Theorem \ref{thm: pure sub discrete iff coincidence} to find
		$k\in\N$ such that $\eta'$ has a coincidence of order $k$.
		Then, Theorem \ref{thm: general estimates ac for one-to-one subs with coincidence} 
		immediately yields concrete bounds for the lower and upper amorphic
		complexity of $(X_{\sub},\sigma)$ depending on $C({\eta'}^k)$ and
		$c({\eta'}^k)$, respectively.
\end{enumerate}

We would like to remark that as a corollary of our results, the amorphic complexity
readily distinguishes constant length substitutions which correspond to different dynamical behavior
according to the next assertion.

\begin{cor}
	Suppose $\sub:A\to\nefw$ is a primitive substitution of constant length.
	Then
	\begin{enumerate}[(i)]
		\item $\ac(\sigma|_{X_\sub})=0$ iff	$(X_{\sub},\sigma)$ is finite;
		\item $1\leq\uac(\sigma|_{X_\sub})\leq\oac(\sigma|_{X_\sub})<\infty$
			iff $(X_{\sub},\sigma)$ has discrete spectrum and is infinite;
		\item $\ac(\sigma|_{X_\sub})=\infty$ iff $(X_{\sub},\sigma)$ has
			partly continuous spectrum.
	\end{enumerate}
\end{cor}
\begin{proof}
	Clearly, the amorphic complexity of  finite subshifts is always zero.
	Further, recall that the amorphic complexity of $(X_{\sub},\sigma)$ is
	infinite if $(X_{\sub},\sigma)$ has partly continuous spectrum, according to
	Corollary \ref{cor:discrete spectrum iff finite sep numbers}.
	By Theorem \ref{thm: general lower and upper bounds for ac}, we have that if
	$(X_{\sub},\sigma)$ has discrete spectrum and is infinite, then
	$1\leq\uac(\sigma|_{X_\sub})\leq\oac(\sigma|_{X_\sub})<\infty$.
	The statement follows.
\end{proof}

\bibliography{lib}
\bibliographystyle{alpha}
\end{document}